\documentclass[reqno,a4paper,12pt]{amsart}

\usepackage[all,poly]{xy}
\usepackage{amsfonts}
\usepackage[mathcal]{eucal}
\usepackage{eufrak}
\usepackage{amssymb}
\usepackage{amsmath}
\usepackage{mathrsfs}
\usepackage{color}
\usepackage[colorlinks]{hyperref}
\usepackage{enumerate}

\definecolor{citecol}{RGB}{145, 1, 1}

\hypersetup{colorlinks=true,citecolor=citecol,linkcolor=blue,linktocpage=true}

\theoremstyle{plain}
\newtheorem {lemma}{Lemma}[section] 
\newtheorem {theorem}[lemma]{Theorem}

\newtheorem {thm}[lemma]{Theorem}
\newtheorem {corollary}[lemma]{Corollary}

\newtheorem {cor}[lemma]{Corollary}

\newtheorem {prop}[lemma]{Proposition}

\theoremstyle{definition}
\newtheorem {remark}[lemma]{Remark}

\newtheorem {example}[lemma]{Example}

\theoremstyle{definition}

\newtheorem{deff}[lemma]{Definition}{}

\newcommand{\M}{\operatorname{\mathbb M}}

\newcommand{\reg}{\operatorname{reg}}

\newcommand{\K}{\mathsf{k}}

\newcommand{\Mod}{\operatorname{\!-Mod}}

\def\M{\mathbb{M}}

\newcommand{\Hom}{\operatorname{Hom}}

\renewcommand{\Im}{\operatorname{Im}}

\begin{document}

\title[Algebras Morita equivalent to weighted Leavitt path algebras]{Unital algebras being Morita equivalent to weighted Leavitt path algebras}


\author{Roozbeh Hazrat}\address{
Centre for Research in Mathematics and Data Science\\
Western Sydney University\\
Australia}
\email{r.hazrat@westernsydney.edu.au}

\author{Tran Giang  Nam}
\address{Institute of Mathematics, VAST, 18 Hoang Quoc Viet, Cau Giay, Hanoi, Vietnam}
\email{tgnam@math.ac.vn}



\begin{abstract}
In this article, we describe the endomorphism ring of a finitely generated progenerator module of a weighted Leavitt path algebra $L_{\K}(E, w)$ of a finite vertex weighted graph $(E, w)$. 
Contrary to the case of Leavitt path algebras, we show that a (full) corner of a weighted Leavitt path algebra is, in general, not isomorphic to a weighted Leavitt path algebra. However, using the above result, we show that   for every full idempotent $\epsilon$ in $L_{\K}(E, w)$, there exists a positive integer $n$ such that $\M_n(\epsilon L_{\K}(E, w) \epsilon)$ is isomorphic to the weighted Leavitt path algebra of a weighted graph explicitly constructed from $(E, w)$.  We then completely describe unital algebras being Morita equivalent to weighted Leavitt path algebras of vertex weighted graphs. In particular, we characterize unital algebras being Morita equivalent to sandpile algebras.

\medskip


\end{abstract}


\maketitle

\section{Introduction} \label{introkai}

The notion of abelian sandpile models was invented in $1987$ by Bak, Tang and Wiesenfeld~\cite{bak}. While defined by a simple local rule, it produces self-similar global patterns that call for an explanation.  The models have been used to describe phenomena such as forest fires, traffic jams, stock market fluctuations, etc. The book of Bak~\cite{bakbook} describes how  events in nature apparently follow this type of behaviour. In~\cite{D1} Dhar systematically associated finite abelian monoid and groups to  sandpile models, now called respectively \emph{sandpile monoids} and \emph{sandpile groups}, and championed the use of them as an invariant which proved to capture many properties of the model. These algebraic structures constitute one of the main themes of the subject and they have been extensively studied (see, e.g., \cite{T, BT,CGGMS} and references there). The abelian sandpile model was independently discovered in $1991$ by combinatorialists Bj\"{o}rner, Lov\'{a}sz and Shor \cite{BLS}, who called it {\it chip-firing}. Indeed, in the last two decades the subject has been enriched by an exhilarating interaction of numerous areas of mathematics, including statistical physics, combinatorics, free boundary PDE, probability, potential theory, number theory and group theory; refer to Levine and Peres's recent survey on the subject \cite{LevinePeres} in more details.

In a different realm, the notion of Leavitt path algebras $L_{\K}(E)$ associated to directed graphs $E$ were introduced by Abrams and Aranda Pino in \cite{aap05}, and independently Ara, Moreno, and Pardo in \cite{amp}. These are a generalisation of algebras introduced by William Leavitt in  1962~\cite{vitt62} as a ``universal'' algebra  $A$ of type $(1,n)$, so that  $A\cong A^{n}$ as right $A$-modules, where $n\in \mathbb N^+$. It was established quite early on in the theory that Leavitt path algebras only produce rings of ranks $(1,n)$. In \cite{vitt62} Leavitt further constructed rings of type $(m,n)$, where $1<m<n$ and showed that these algebras surprisingly are domains.  When $m\ge 2$, this universal ring is not realizable as a Leavitt path algebra. With this in mind, the notion of {\it weighted}  Leavitt path algebras $L_{\K}(E, w)$ associated to weighted graphs $(E, w)$  were introduced by the first author in \cite{H} (see \cite{Pre} for a nice overview of this topic).  The weighted Leavitt path algebras provide a natural context in which all of Leavitt's algebras (corresponding to any pair $m,,n \in \mathbb{N}$) can be realised as a specific example.

The study of the commutative monoid $\mathcal{V}(A)$ of isomorphism classes of finitely generated projective left modules of a unital ring $A$ (with operation $\oplus$) goes back to the work of Grothendieck and Serre.  For a Leavitt path algebra $L_{\K}(E)$, the monoid $\mathcal{V}(L_{\K}(E))$ has received substantial attention since the introduction of the topic. Furthermore, the monoid $\mathcal{V}(L_{\K}(E, w))$ has been described completely by the first author in \cite{H}, and subsequently by Preusser in \cite{P}.  
In~\cite{GeneRooz} Abrams and the first author established that these monoids could be naturally related to the sandpile monoids of graphs. This relationship allows us to associate an algebra, a sandpile algebra (see \cite[Definition 6.11]{GeneRooz}), to the theory of sandpile models, thereby opening up an avenue by which to investigate sandpile models via the structure of weighted Leavitt path algebras (particularly the sandpile algebras), and vice versa. 

In this article, we investigate the relations between these objects by characterizing unital algebras which are Morita equivalent to sandpile algebras. More precisely, we completely describe unital algebras which are Morita equivalent to weighted Leavitt path algebras of finite vertex weighted graphs--containing the class of sandpile algebras as a special example.

The article is organized as follows. In Section 2, for the reader's convenience, we collect some basic definitions and facts on weighted graphs and weighted Leavitt path algebras, as well as provide a description of the quotient of a weighted Leavitt path algebra by an ideal generated a set of vertices (Theorem \ref{quotientwLPa}) and the matrix ring over a weighted Leavitt path algebra (Theorem \ref{matrixalgeoverwLPA}).

In Section 3, based on Theorem \ref{quotientwLPa} and the structure of the monoid $\mathcal{V}(L_{\K}(E, w))$ established by the first author in \cite[Theorem 5.21]{H}, we completely describe the endomorphism ring of a projective and finitely generated generator for the category of left modules over the  weighted Levitt path algebra $L_{\K}(E, w)$ of a finite vertex weighted graph $(E, w)$ (Theorem \ref{maintheo1}). Consequently, this yields that for every full idempotent $\epsilon$ in $L_{\K}(E, w)$, there exists a positive integer $n$ such that $\M_n(\epsilon L_{\K}(E, w) \epsilon)$ is isomorphic to the weighted Leavitt path algebra of a finite vertex weighted graph explicitly constructed from $(E, w)$ (Corollary \ref{CornerofwLPAs}). Also, we show that a corner of a weighted Leavitt path algebra of a finite vertex weighted graph is, in general, not a weighted Leavitt path algebra (Example \ref{CornerofwLPAsexample}). We should mention every corner of a Leavitt path algebra of a finite graph is always isomorphic to a Leavitt path algebra (see \cite[Theorem 3.15]{an:colpaofgalpa}), and a corner of a Leavitt path algebra of an arbitrary graph need not in general be isomorphic to a Leavitt path algebra (see \cite[Example 2.11 and Corollary 3.8]{adn:rcolpaasawcctg}).
As a consequence of Theorem \ref{matrixalgeoverwLPA} and \ref{CornerofwLPAs}, we obtain that  a unital $\K$-algebra $A$ is Morita equivalent to the weighted Leavitt path algebra $L_K(E, w)$ of a finite vertex weighted graph $(E, w)$ if and only if there exists a positive integer $n$ such that $\M_n(A)$ is isomorphic to the weighted Leavitt path algebra of a finite vertex weighted graph explicitly constructed from $(E, w)$ (Theorem \ref{maintheo2}). In particular, we completely describe unital algebras being Morita equivalent to sandpile algebras (Corollary \ref{sandpilealgebra}). 

 
Throughout we write $\mathbb N$ for the set of non-negative integers, and $\mathbb N^+$ for the set of positive integers.

 \section{Weighted Leavitt path algebras} \label{weightedgraphssec}
In this section, we give a description of the quotient of a weighted Leavitt path algebra by an ideal generated a set of vertices (Theorem \ref{quotientwLPa}) and the matrix ring over a weighted Leavitt path algebra (Theorem \ref{matrixalgeoverwLPA}).\medskip

A \emph{(directed) graph} is a quadruple $E=(E^0,E^1,s,r)$, where $E^0$ and $E^1$ are sets and $s,r:E^1\rightarrow E^0$ are maps.   The elements of $E^0$ are called 
\emph{vertices} and the elements of $E^1$ \emph{edges}.  If $e$ is an edge, then $s(e)$ is called its \emph{source} and $r(e)$ its \emph{range}. If $v$ is a vertex and $e$ an edge, we say that $v$ {\it emits} $e$ if $s(e)=v$, and $v$ {\it receives} $e$ if $r(e)=v$.   A graph $E$ is called \emph{row-finite} if any vertex in $E$ emits a finite number (possibly zero) of edges. The graph $E$ is called  \emph{finite} if $E^0$ and $E^1$ are finite sets. 
A vertex is called a {\it sink} if it emits no edges.  
A vertex is called \emph{regular} if it is not a sink and does not emit infinitely many edges.  The subset of $E^0$ consisting of all the regular vertices is denoted by $E^0_{\reg}$. A vertex is called a {\it course} if it receives no edges. 

A  {\it path}  $p$ in $E$ is a sequence $p=e_{1}e_{2}\cdots e_{n}$ of edges in $E$ such that $r(e_{i})=s(e_{i+1})$ for $1\leq i\leq n-1$. In this case, we say that the path $p$ starts at the vertex $s(p) := s(e_{1})$ and and ends at the vertex $r(p) :=r(e_{n})$, and has  \emph{length} $|p|:=n$. We assign the length zero to vertices.  A \emph{closed path} (based at $v$) is a  path $p$ such that $s(p)=r(p)=v$. A \emph{cycle} (based at $v$) is a closed path $p=e_1 e_2 \cdots e_n$ based at $v$ such that $s(e_i)\neq s(e_j)$ for any $i\neq j$. A cycle $c$ is called a {\it loop} if $|c| =1$.
 
 
 
 A subset $H \subseteq E^0$ is said to be \emph{hereditary} if
for any $e \in E^1$,  $s(e)\in H$ implies $r(e)\in H$. 
 For a hereditary subset $H$ of $E$, we define the quotient graph $E/H$ as follows:
  $$(E/ H)^0=E^0\setminus H^0   \ \ \mbox{and} \ \  (E/H)^1=\{e\in E^1\;| \ r(e)\notin H\}.$$ The source and range maps of $E/H$ are the source and range maps  restricted from the graph $E$.


\begin{deff}\label{weightedgraphdef}
A \emph{weighted graph} is a pair $(E,w)$, where $E$ is a graph and $w:E^1\rightarrow \mathbb N^+$ is a map. If $e\in E^1$, then $w(e)$ is called the \emph{weight} of $e$. A weighted graph $(E,w)$ is called  \emph{row-finite} if the graph $E$ is row-finite. A weighted graph $(E,w)$ is called  \emph{finite} if the graph $E$ is finite.

For each regular vertex $v$ in a weighted graph $(E,w)$ we set $w(v):=\max\{w(e)\mid e\in s^{-1}(v)\}$.  This gives  a map (called $w$ again) $w:E_{\reg}^0\rightarrow \mathbb N^+$. 

A row-finite weighted graph $(E,w)$ is called a \emph{vertex weighted graph} if for any $v\in E^0_{\reg}$, $w(e) = w(e')$ for all  $e,e' \in s^{-1}(v)$.     

A vertex weighted graph is called a \emph{balanced weighted graph} if  $w(v)=|s^{-1}(v)|$ for all $v\in E_{\reg}^0$.   
\end{deff}


It is worth mentioning the following note.

\begin{remark}\label{weightonquotientremark}
Let $(E,w)$ be a row-finite weighted graph and $H$ a hereditary subset of $E^0$. We form the quotient graph $E/H$. Then, there is a natural way to define a weight function $w_r$ on $E/H$ by restricting the weight function on $E$ to $E/H$. 
\end{remark}

We next recall the notion of weighted Leavitt path algebras. These are algebras associated to row-finite weighted graphs. We refer the reader to \cite{H} and \cite{Pre} for a detailed analysis of these algebras. 
\begin{deff}\label{weighteddef}
Let $(E,w)$ be a  row-finite weighted graph and $\K$ a field.  The free $\K$-algebra generated by $\{v,e_i,e_i^*\mid v\in E^0, e\in E^1, 1\leq i\leq w(e)\}$ subject to relations
\begin{enumerate}[(i)]
		\item $uv=\delta_{uv}u,  \text{ where } u,v\in E^0$,
		\medskip
		\item $s(e)e_i=e_i=e_ir(e),~r(e)e_i^*=e_i^*=e_i^*s(e),  \text{ where } e\in E^1, 1\leq i\leq w(e)$,
		\medskip
		\item 
		$\sum_{e\in s^{-1}(v)}e_ie_j^*= \delta_{ij}v, \text{ where } v\in E_{\reg}^0 \text{ and } 1\leq i, j\leq w(v)$, 
		\medskip 
		\item $\sum_{1\leq i\leq w(v)}e_i^*f_i= \delta_{ef}r(e), \text{ where } v\in E_{\reg}^0 \text{ and } e,f\in s^{-1}(v)$,
\end{enumerate}
is called the {\it weighted Leavitt path algebra} of $(E,w)$ over $\K$, and denoted $L_\K(E,w)$, where $\delta$ is the Kronecker delta. In relations (iii) and (iv) we set $e_i$ and $e_i^*$ to be zero whenever $i > w(e)$.    
\end{deff}

Note that if the weight of each edge is $1$, then $L_\K(E,w)$ reduces to the usual Leavitt path algebra $L_\K(E)$  of the graph $E$.  It is easy to see that the mappings given by $v\longmapsto v$ for all $v\in E^0$, $e_i\longmapsto e^*_i$ and $e^*_i\longmapsto e_i$ for all $e\in E^1$ and $1\le i\le w(e)$, produce an involution on the $\K$-algebra $L_{\K}(E, w)$. Also, if the weighted graph $(E, w)$ is finite, then $L_\K(E,w)$ is a unital ring having identity $1_{L_\K(E,w)} = \sum_{v\in E^0}v$ (see \cite[Proposition 5.7 (2)]{H}). Moreover, $L_{\K}(E, w)$ has the following property: if $\mathcal{A}$ is a $\K$-algebra generated by a family of elements $\{a_v, b_e, c_{e^*}\mid v\in E^0, e\in E^1, 1\leq i\leq w(e)\}$ satisfying the relations analogous to (1) - (4)  in Definition~\ref{weighteddef}, then there is a unique $\K$-algebra homomorphism $\varphi: L_{\K}(E)\longrightarrow \mathcal{A}$ given by $\varphi(v) = a_v$, $\varphi(e) = b_e$ and $\varphi(e^*) = c_{e^*}$.  We will refer to this property as the Universal Property of $L_{\K}(E, w)$.

In \cite{HR} Hazrat and Preusser found a Gr\"{o}bner-Shirshov basis of a weighted Leavitt path algebra. We will represent again this result in Theorem \ref{GSbasis} below. To do so, we need to recall some notions. Given a row-finite weighted graph $(E, w)$ let $s(e_i) := s(e)$, $r(e_i) = r(e)$, $s(e^*_i) = r(e)$ and $r(e^*_i) = s(e)$ for all $e\in E^1$ and $1\le i\le w(e)$. Let $\langle X\rangle$ denote the set of all nonempty words over $X := \{v, e_i, e^*_i\mid v\in E^0, e\in E^1, 1\le i\le w(e)\}$. A word $p\in \langle X\rangle$ is called a {\it generalized path} in $(E, w)$ if either $p = x_1\cdots x_n$ for some $n\ge 1$ and $x_1, \hdots, x_n \in \{e_i, e_i^*\mid e\in E^1, 1\le i \le w(e)\}$ such that $r(x_k) = s(x_{k+1})$ for all $1\le k\le n-1$ or $p= x_1$ for some $x_1\in E^0$.

For every $v\in E_{\reg}^0$ fix an $\alpha^v\in s^{-1}(v)$ such that $w(\alpha^v)=w(v)$. The words 
\begin{center}
$\alpha^v_i(\alpha^v_j)^*$ ($v\in E_{\reg}^0, \ 1\le i, j\le w(v)$)	and $e_1^*f_1$ ($v\in E_{\reg}^0, \ e, f\in s^{-1}(v)$)
\end{center}
are called {\it forbidden}. A generalised path is called {\it normal} if none of its subword  is forbidden.

\begin{theorem}[{\cite[Theorem 16]{HR}}]\label{GSbasis}
Let $(E, w)$ be a row-finite weighted graph and $\K$ a field. Then the weighted Leavitt path algebra $L_{\K}(E, w)$ has a basis consisting of normal generalized paths.
\end{theorem}

Let $(E, w_E)$ be a row-finite weighted graph. A {\it weighted subgraph} $X$ of the weighted graph $X = (X^0, X^1, r_X, s_X, w_X)$ such that $X^0\subseteq E^0$, $X^1\subseteq E^1$, and $r_X, s_X$ are respectively the restrictions of $s_E, r_E$ on $X^1$ and $w_X(e) = w_E(e)$ for all $e\in X^1$. A weighted subgraph $X$ is called {\it complete} if $s^{-1}_X(v) = s^{-1}_E(v)$ for all $v\in X^0_{\reg}$.

As a consequence of Theorem \ref{GSbasis}, we have the following useful note.

\begin{cor}\label{completewsubgraph}
Let $\K$ be a field, $(E, w)$ a row-finite weighted graph and $(X, w)$	a complete weighted subgraph of $(E, w)$. Then the map $L_{\K}(X, w)\longrightarrow L_{\K}(E, w)$, defined by $v\longmapsto v$ for all $v\in X^0$,  $e_i\longmapsto e_i$ and $e^*_i\longmapsto e^*_i$ for all $e\in X^1$ and $1\le i\le w(e)$, is an injective homomorphism of $\K$-algebras.
\end{cor}
\begin{proof}
By the Universal Property of $L_{\K}(X, w)$, there exists a $\K$-algebra homomorphism $\phi: L_{\K}(X, w)\longrightarrow L_{\K}(E, w)$ such that $\phi(v) = v$ for all $v\in X^0$,  $\phi(e_i) = e_i$ and $\phi(e^*_i) = e^*_i$ for all $e\in X^1$ and $1\le i\le w(e)$.	Since $(X, w)$ is a complete weighted subgraph of $(E, w)$, it follows that every normal generalized path in $X$ is also a normal generalized path in $E$. From this note and Theorem \ref{GSbasis}, we immediately obtain that $\phi$ is injective, thus finishing the proof.	
\end{proof}	


For any unital ring $R$, $\mathcal{V}(R)$ denotes the set of isomorphism classes (denoted by $[P]$) of finitely generated projective left $R$-modules. $\mathcal{V}(R)$ is an abelian monoid with operation \[[P] + [Q] = [P\oplus Q]\] for any isomorphism classes $[P]$ and $[Q]$. On the other hand, for any row-finite vertex weighted graph $(E, w)$ the monoid $M_{(E, w)}$ is defined as follows. Denote by $\mathbb F_E$ the free abelian
monoid (written additively) with generators $E^0$, and define relations on $\mathbb F_E$ by setting 
\begin{equation}\label{meinpeace}
w(v)v = \sum_{e \in s^{-1}(v)}r(e), 
\end{equation}
for every $v\in E_{\reg}^0$. Let $\sim_{E}$ be the congruence relation on $\mathbb F_E$ generated by these relations. Then $M_{(E, w)}$ is defined to be the quotient monoid
$\mathbb F_E/\sim_{E}$; we denote an element of $M_{(E, w)}$ by $[x]$, where $x\in \mathbb F_E$. 

\begin{thm}[{\cite[Theorem 5.21]{H}} and \cite{P}]\label{wgraphmonoid}
Let $(E, w)$ be a finite vertex weighted graph and $\K$ a field. Then the map $[v]\longmapsto [L_{\K}(E, w)v]$ yields an isomorphism of abelian monoids $M_{(E, w)}\cong \mathcal{V}(L_{\K}(E, w))$. Specifically, these two useful consequences follow immediately.

$(1)$ For any $v\in E_{\reg}^0$, $(L_{\K}(E, w)v)^{w(v)}\cong \bigoplus_{v\in s^{-1}(v)}L_{\K}(E, w)r(e)$ as left $L_{\K}(E, w)$-modules;

$(2)$ For any nonzero finitely generated projective left $L_{\K}(E, w)$-module $P$, there exists a subset of (distinct) vertices $W$ of $E^0$ and positive integers $\{n_w\mid w\in W\}$ for which $P\cong \bigoplus_{w\in W}(L_{\K}(E, w)w)^{n_w}$.
\end{thm}

Next we give the notion of saturated subsets of weighted graphs.

\begin{deff}\label{satuset}
Let $(E, w)$ be a row-finite weighted graph. If $v$ in $E^0$ and $1\le i\le w(v)$, denote by $A_{v,i}$ the set $\{e\in s^{-1}(v)\mid w(e)\ge i\}$. A subset $H$ of $E^0$ is called {\it saturated} if whenever $v\in E_{\reg}^0$ with the property that $\{r(e)\mid e\in A_{v,i}\}\subseteq H$ for some $1\le i\le w(v)$, then $v\in H$.
\end{deff}	

We should mention the following useful note.

\begin{remark}\label{satusetrem}
For any vertex weighted graph $(E, w)$, a subset $H$ of $E^0$ is  saturated if and only if whenever $v\in E_{\reg}^0$ with the property that $r(s^{-1}(v))\subseteq H$, then $v\in H$. 	
\end{remark}

For clarification, we illustrate Definition \ref{satuset} by presenting the following example.

\begin{example}
Let $(E, w)$ be the weighted graph
$$\xymatrix{\bullet^{u}&\bullet^{v}\ar[l]_{e} \ar[r]^{f}& \bullet^{x}},$$ where $w(e) =1$ and $w(f) =2$. We then have $A_{v,2} =\{f\}$ and that $H :=\{x\}$ is not a saturated subset of $E^0$ (since $v\notin H$ and $\{r(e)\mid e\in A_{v,2}\} = \{x\}= H$).
\end{example}

The following theorem provides a description of the quotient of a weighted Leavitt path algebra by an ideal generated a set of vertices, which is a generalization of \cite[Theorem 2.4.15]{TheBook}.

\begin{theorem}\label{quotientwLPa}
Let $\K$ be a field, $(E,w)$ a row-finite weighted graph, and $H$ a hereditary and saturated subset of $E^0$. Then the following statements hold:

$(1)$  There exists an isomorphism of $\K$-algebras	\[\bar{\phi}: L_{\K}(E, w)/I(H) \longrightarrow L_{\K}(E/H, w_r),\] where $I(H)$ is the ideal of $L_{\K}(E, w)$ generated by $H$.

$(2)$ $I(H) \cap E^0 = H$.
\end{theorem}
\begin{proof}
(1) We define the elements $\{p_v\mid v\in E^0\}$ and $\{s_{e_i}, s_{e^*_i}\mid e\in E^1, 1\le i\le w(e)\}$ of $L_{\K}(E/H, w_r)$ by setting:
	\begin{equation*}
p_{v}=  \left\{
\begin{array}{lcl}
v&  & \text{if } v \notin H\\
0&  & \text{otherwise},  
\end{array}%
\right.
\end{equation*}%

\begin{equation*}
s_{e_i}=  \left\{
\begin{array}{lcl}
e_i&  & \text{if } r(e) \notin H\\
0&  & \text{otherwise},  
\end{array}%
\right.
\end{equation*}%
\medskip
and 
\begin{equation*}
s_{e^*_i}=  \left\{
\begin{array}{lcl}
s_{e^*_i}&  & \text{if } r(e) \notin H\\
0&  & \text{otherwise}. 
\end{array}%
\right.
\end{equation*}%
We claim that $\{p_v, s_{e_i}, s_{e^*_i}\mid v\in E^0, e\in E^1\}$ is a family in $L_{\K}(E/H, w_r)$ satisfying the relations analogous to (1) - (4) in Definition~\ref{weighteddef}. For (1) and (2), this is a straightforward computation done by cases, with the only nontrivial situation arising when $e\in (E/H)^1$. But then $r(e)\notin H$, and hence $s_{e_i} - s_{e_i}p_{r(e)} = e_i - e_i r(e) = 0$ and $s_{e^*_i} - p_{r(e)}s_{e^*_i} = e^*_i - r(e)e^*_i  = 0$ 
in $L_{\K}(E/H, w_r)$. Since $r(e)\notin H$ and $H$ is hereditary, we have $s(e)\notin H$, and so $s_{e_i} - p_{s(e)}s_{e_i} = e_i - s(e)e_i = 0$ and $s_{e^*_i} - s_{e^*_i}p_{s(e)} = e^*_i- e^*_i s(e) = 0$ in $L_{\K}(E/H, w_r)$.

For (4), that $\sum_{1\leq i\leq w(v)}s_{e_i^*}s_{f_i}- \delta_{ef}p_{r(e)} =0$ in $L_{\K}(E/H, w_r)$, $\text{ where } v\in E_{\reg}^0 \text{ and } e,f\in s^{-1}(v)$, is straightforward.

For (3), let $v$ be a regular vertex in $E$. For $1\le i\neq j\le w(v)$, that $\sum_{e\in s^{-1}(v)}s_{e_i}s_{e_j^*}  =0$ in $L_{\K}(E/H, w_r)$ is straightforward. For $1\le i\le w(w)$, we show that $\sum_{e\in s^{-1}(v)}s_{e_i}s_{e_i^*} = p_v$ in $L_{\K}(E/H, w_r)$. Consider the following two cases:

{\it Case} $1$: $v\in H$. We then have $p_v =0$ in $L_{\K}(E/H, w_r)$. Since $H$ is hereditary, it follows that $r(e)\in H$ for all $e\in s^{-1}(v)$, and so $\sum_{e\in s^{-1}(v)}s_{e_i}s_{e_i^*} =0 = p_v$ in $L_{\K}(E/H, w_r)$, as desired.

{\it Case} $2$: $v\notin H$. We then have $p_v = e$ in $L_{\K}(E/H, w_r)$. Since $H$ is saturated, there exists an edge $e\in s^{-1}(v)$ such that $w(e)\ge i$ and $r(e)\notin H$, and therefore we have a third relation $L_{\K}(E/H, w_r)$ at $v$:
\[v = \sum_{e\in s_{E/H}^{-1}(v)}e_ie_i^*.\] Consider $e\in E^1$ with $s(e) =v$. If $e\in (E/H)^1$, then $s_{e_i}s_{e^*_i} = e_ie^*_i$. If $e\notin (E/H)^1$, then $s_{e_i}s_{e^*_i} = 0$. Hence we have \[\sum_{e\in s^{-1}(v)}s_{e_i}s_{e_i^*} = \sum_{e\in s_{E/H}^{-1}(v)}s_{e_i}s_{e_i^*} = v = p_v,\] as desired. So by the Universal Property of $L_{\K}(E, w)$, there is a unique $\K$-algebra homomorphism $L_{\K}(E, w) \longrightarrow L_{\K}(E/H, w_r)$ such that $\phi(v) = p_v$ for all $v\in E^0$, and $\phi(e_i) = s_{e_i}$ and $\phi(e^*_i) = s_{e^*_i}$ for all $e\in E^1, 1\le i\le w(e)$. By the definition of $E/H$, it is obvious that $\phi$ is surjective and $I(H)\subseteq \ker(\phi)$.  Consequently, there exists an induced $\K$-algebra homomorphism \[\bar{\phi}: L_{\K}(E, w)/I(H) \longrightarrow L_{\K}(E/H, w_r),\] which is surjective. 

We now define a map $\varphi: L_{\K}(E/H, w_r) \longrightarrow L_{\K}(E, w)/I(H)$ by setting: $\varphi(v) = v + I(H)$ for all $v\in (E/H)^0$, $\varphi(e_i) = e_i + I(H)$ and $\varphi(e^*_i) = e^*_i + I(H)$ for all $e\in (E/H)^1$ and $1\le i\le w_r(e)$. By the Universal Property of $L_{\K}(E/H, w_r)$, $\varphi$ extends to a $\K$-algebra homomorphism. It is then  straightforward to see that $\varphi \bar{\phi}$ gives the identity on the canonical generators, and so $\varphi \bar{\phi} = id_{L_{\K}(E, w)/I(H)}$. This implies that $\bar{\phi}$ is injective, and therefore it is an isomorphism.

(2) It is obvious that $H \subseteq I(H)\cap E^0$. Conversely, let $v\in E^0\setminus H$. Let $\bar{\phi}$ be the $\K$-algebra isomorphism given in item (1). We then have that $\bar{\phi}(v)$ is a nonzero element in $L_{\K}(E/H, w_r)$ by Theorem \ref{GSbasis}, and therefore $v\notin I(H)$, thus finishing the proof.
\end{proof}	

As an application of Theorem \ref{quotientwLPa}, we provide unital $\K$-algebras of type $(m, n)$ ($1< m < n$) but not isomorphic to the Leavitt algebras $L_{\K}(m, n)$.

\begin{example}
Let $(E,w)$ be the weighted graph

$$
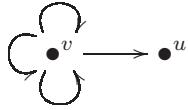
\xymatrix{\bullet^{v} \ar@(ul,ur) \ar@(lu,ld) \ar@(dl,dr) \ar[r]& \bullet^{u}}$$ 

\bigskip 
\noindent with all the edges have weight two. Consider the weighted Leavitt path algebra $A:= L_\K(E,w)$. We know $\mathcal V(L_\K(E,w)) \cong M_{(E,w)}$, where  by (\ref{meinpeace}),  $M_{(E,w)}= \langle u, v\rangle / \langle 2v=3v+u \rangle $. Since the order unit  order unit $[L_\K(E,w)]$ in $\mathcal V(\mathcal (L_\K(E,w))$ is presented with $u+v$ in $M_{(E,w)}$, a quick computation shows $A^2\cong A^3$, so the typle of $L_\K(E,w)$ is $(2,3)$. Furthermore, considering the (graded) ideal $I(u)$ generated by hereditary and saturated subset $\{u\}$. As an application of Theorem~\ref{quotientwLPa}, we have $$L_\K(E,w)/I(u) \cong L_\K(2,3),$$ which is again of type $(2,3)$. 
	
\end{example}

In the remainder of this section, we describe matrix algebras over weighted Leavitt path algebras. To do so, we need the following useful notion introduced in \cite[Definition 9.1]{AT} for the case of graphs.

\begin{deff} \label{matrixextdef}    Let $(E, w)$ be a row-finite weighted graph and $n$ a positive integer.   Let $(\M_nE, \bar{w})$  be the weighted graph formed from $(E, w)$ by taking each $v\in E^0$  and attaching a ``head" of length $n - 1$ of the form 
	$$\xymatrix{  \bullet^{v_{n-1}} \ar[r]^{e^v_{n-1}} & \cdots\ar[r]^{e_3^v}& \bullet^{v_2} \ar[r]^{e_2^v} & \bullet^{v_1} \ar[r]^{e^v_1} &v }  \ $$
to $E$, where $\bar{w}$ is an extension of $w$ such that $\bar{w}(e^v_i) =1$ for all $v\in E^0$ and $1\le i\le n-1$.
\end{deff}

For clarification, we consider the following example.

\begin{example}\label{matrixextexam}
	If $E$ is the graph
	
	$$
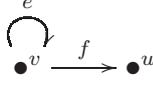
\xymatrix{\bullet^{v} \ar@(ul,ur)^e\ar[r]^f& \bullet^{u}}$$ then
	$\M_2E$ is the graph
	
	$$
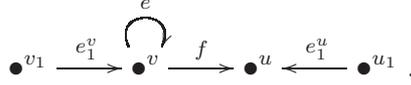
\xymatrix{\bullet^{v_1}\ar[r]^{e^v_1}&
		\bullet^{v} \ar@(ul,ur)^e\ar[r]^f& \bullet^{u}&\bullet^{u_1}\ar[l]_{e^u_1}}.$$ 
\end{example}

We end this section by showing that the matrix algebra over a weighted Leavitt path algebra is again a weighted Leavitt path algebra, which is a generalization of \cite[Proposition 9.3]{AT}.

\begin{theorem}\label{matrixalgeoverwLPA}
Let $(E, w)$ be a row-finite weighted graph, $\K$ a field and $n$ a positive integer. Then $L_{\K}(\M_nE, \bar{w})\cong \M_n(L_{\K}(E, w))$ as $\K$-algebras.
\end{theorem}
\begin{proof}
We define the elements $\{p_v\mid v\in (\M_nE)^0\}$ and $\{s_{e_i}, s_{e^*_i}\mid e\in (\M_nE)^1, 1\le i\le \bar{w}(e)\}$ of $\M_n(L_{\K}(E, w))$ by setting:	for each $v\in E^0$ and $e\in E^1$ define
\begin{center}
$p_v = vE_{1, 1}$,\quad $s_{e_i} = e_iE_{1, 1}$\quad and\quad $s_{e^*_i} = e^*_iE_{1, 1}$\quad ($1\le i \le w(e)$).	
\end{center} 
Also for $v\in E^0$ and $k\in \{1, 2, \hdots, n-1\}$ define
\begin{center}
	$p_{v_k} = vE_{k+1, k+1}$,\quad $s_{e^v_k} = vE_{k+1, k}$\quad and\quad $s_{(e^v_k)^*} = vE_{k, k+1}$,	
\end{center} 
where we let $x E_{i,j}$ $(x \in L_{\K}(E, w))$ denote the matrix in $\M_n(L_{\K}(E, w))$ with $x$ in the $(i, j)^{\text{th}}$ position and $0$'s elsewhere.

It is straightforward to show that $$\{p_v, p_{v_k} s_{e_i}, s_{e^*_i}, s_{e^v_k}, s_{(e^v_k)^*}\mid v\in E^0, e\in E^1,  1\le k\le n-1, 1\le i\le w(e)\}$$ is a family in $\M_n(L_{\K}(E, w))$ satisfying the relations analogous to (1) - (4) in Definition~\ref{weighteddef}. So by the Universal Property of $L_{\K}(\M_nE, \bar{w})$, there is a unique $\K$-algebra homomorphism $\theta: L_{\K}(\M_nE, \bar{w}) \longrightarrow \M_n(L_{\K}(E, w))$ such that $\theta(v) = p_v$, $\theta(v_k) = p_{v_k}$, $\theta(e_i) = s_{e_i}$, $\phi(e^*_i) = s_{e^*_i}$, $\theta(e^v_k) = s_{e^v_k}$ and $\theta((e^v_k)^*) = s_{(e^v_k)^*}$ for all $v\in E^0$, $e\in E^1$, $1\le k\le n-1$ and $1\le i\le w(e)$.

We claim that $\theta$ is surjective by following the corresponding proof of \cite[Proposition 9.3]{AT}. It is sufficient to show that $vE_{l, t}$, $e_iE_{l, t}$ and $e^*_iE_{l, t}$ are in $\Im(\theta)$ for all $v\in E^0$, $e\in E^1$, $1\le i\le w(e)$ and $1\le l, t\le n$. For $l= t$ we have 
	\begin{equation*}
vE_{l, t}= vE_{l, l} = \left\{
\begin{array}{lcl}
p_v&  & \text{if } l =1\\
p_{v_{i-1}}&  &\text{if } l\ge 2,  
\end{array}%
\right.
\end{equation*}%
for $l\ge t+1$ we have \[vE_{l, t}=(vE_{l, l-1})(vE_{l-1, l-2})\cdots (vE_{t+1, t})= s_{e^v_{l-1}}s_{e^v_{l-2}}\cdots s_{e^v_{t}},\]
for $t\ge l+1$ we have 
\[vE_{l, t}=(vE_{l, l+1})(vE_{l+1, l+2})\cdots (vE_{t-1, t})= s_{(e^v_{l})^*}s_{(e^v_{l+1})^*}\cdots s_{(e^v_{t-1})^*}.\]
For every $e\in E^1$, $1\le i\le w(e)$ and $1\le l, t\le n$ we have
\[e_iE_{l, t}=(s(e)E_{l, 1})(e_iE_{1, 1}) (r(e)E_{1, t})= s_{e^{s(e)}_{l-1}}s_{e^{s(e)}_{l-2}}\cdots s_{e^{s(e)}_{1}}s_{e_i} s_{(e^{r(e)}_{1})^*}s_{(e^{r(e)}_{l+1})^*}\cdots s_{(e^{r(e)}_{t-1})^*}\] 
and \[e^*_iE_{l, t}=(r(e)E_{l, 1})(e^*_iE_{1, 1}) (s(e)E_{1, t})= s_{e^{r(e)}_{l-1}}s_{e^{r(e)}_{l-2}}\cdots s_{e^{r(e)}_{1}}s_{e^*_i} s_{(e^{s(e)}_{1})^*}s_{(e^{s(e)}_{l+1})^*}\cdots s_{(e^{s(e)}_{t-1})^*},\] as desired.

We next show that $\theta$ is injective. Indeed, let $x$ be an arbitrary element of $\ker(\theta)$. By Theorem \ref{GSbasis}, it follows that $x$ can be written uniquely in the form $x = \sum^m_{i=1}k_ip_i$, where $k_i\in \K$ and $p_i$ is a normal generalized path in $(\M_nE, \bar{w})$. We note that for very $v\in E^0$ we have $(e^v_i)^* e^v_j = \delta_{ij}v_{i-1}$ in $L_{\K}(\M_nE, \bar{w})$ for all $1\le i, j\le n-1$, where $v_0 := v$. Therefore, for each $1\le i\le m$ we have $p_i = \alpha_iq_i\beta^*_i$, where $q_i$ is a  normal generalized path in $(E, w)$, and both $\alpha_i$ and $\beta_i$ are paths in $\M_nE$ of the form: either $e^v_{k}\cdots e^v_2e^v_1$ for some $k\ge 1$ and $v\in E^0$ or $v$ for some $v\in E^0$. We then have $x = \sum^m_{i=1}k_i\alpha_i q_i \beta^*_i$ and $$0= \theta(x) = \sum^m_{i=1}k_i\theta(\alpha_i q_i \beta^*_i) = \sum^m_{i=1}k_i\theta(\alpha_i) \theta(q_i) \theta(\beta^*_i).$$
On the other hand, we have 
\begin{center}
$\theta(\alpha_i) = r(\alpha_i)E_{|\alpha_1|+1, 1}$, $\theta(q_i) = q_iE_{1, 1}$ and $\theta(\beta^*_i) = r(\beta_i)E_{1, |\beta_1|+1},$	
\end{center}
so $$\theta(\alpha_i) \theta(q_i) \theta(\beta^*_i)= q_i E_{|\alpha_i+1, |\beta_i|+1} \text{ and } \theta(x) = \sum^m_{i=1}k_iq_i E_{|\alpha_i|+1, |\beta_i|+1} = 0.$$

Write $\{1, 2, \hdots, m\} = V_1\sqcup V_2\sqcup \cdots \sqcup V_t$ such that $(|\alpha_i|, |\beta_i|) = (|\alpha_j|, |\beta_j|)$ for all $i, j\in V_k$ and $1\le k\le m$, and $(|\alpha_i|, |\beta_i|) \neq (|\alpha_j|, |\beta_j|)$ for all $i\in V_k$, $j\in V_l$ and $1\le k\neq l\le m$. We then have \[0=\sum^m_{i=1}k_iq_i E_{|\alpha_i|+1, |\beta_i|+1} = \sum^t_{j=1}(\sum_{i\in V_j}k_iq_i)E_{|\alpha_i|+1, |\beta_i|+1}\] in $\M_n(L_K(E, w))$, and so $\sum_{i\in V_j}k_iq_i = 0$ for all $1\le j\le t$. Since all the elements $q_i$ are normal generalized path in $(E, w)$ and by Theorem \ref{GSbasis}, we must have $k_i = 0$ for all $1\le i \le m$, that means, $x = 0$. This implies that $\theta$ is injective, and so $\theta$ is an isomorphism, thus finishing the proof.
\end{proof}



\section{Corners of unital weighted Leavitt path algebras}\label{wLpasection}
In this section, we describe the endomorphism ring of a finitely generated progenerator (i.e., finitely generated projective and a generator)
 for the category of left modules over the  weighted Levitt path algebra $L_{\K}(E, w)$ of a finite vertex weighted graph $(E, w)$ (Theorem \ref{maintheo1}). Consequently, this yields that for every full idempotent $\epsilon$ in $L_{\K}(E, w)$, there is a positive integer $n$ such that $\M_n(\epsilon L_{\K}(E, w) \epsilon)$ is a weighted Leavitt path algebra (Corollary \ref{CornerofwLPAs}). Also, we show that a corner of a weighted Leavitt path algebra is, in general, not a weighted Leavitt path algebra (Example \ref{CornerofwLPAsexample}). Finally, we describe unital algebras and sandpile algebras that are Morita equivalent to weighted Leavitt path algebras (Theorem \ref{maintheo2} and Corollary \ref{sandpilealgebra}).\medskip

We begin this section by describing hereditary saturated closures of subsets of vertices of of a vertex weighted graph.
Let $(E, w)$ be a row-finite weighted graph and $X$ a nonempty subset of $E^0$. Let $H(X) := r(s^{-1}(X))$ and $S(X) := \{v\in E_{\reg}^0\mid r(s^{-1}(v))\subseteq X\}$. We now define $G_0 := X$ and  for $n\ge 0$ we define inductively $G_{n+1} := H(G_n)\cup S(G_n)\cup G_n$, and we define $\overline{X}= \bigcup_{n\ge 0}G_n$.

\begin{lemma}\label{h-sclosure}
Let $(E, w)$ be a vertex weighted graph and $X$ a nonempty subset of $E^0$. Then $\overline{X}$ is the smallest hereditary and saturated subset of $E^0$ containing $X$. 
\end{lemma}	
\begin{proof}
By Remark \ref{satusetrem} and repeating the proof of \cite[Lemma 2.0.7]{TheBook} we immediately obtain the statement.	
\end{proof}	

The following important proposition gives a description of projective and finitely generated generators for the category of left modules over the weighted Leavitt path algebra of a finite vertex weighted graph.

Recall from Theorem~\ref{wgraphmonoid} that for any nonzero finitely generated projective left $L_{\K}(E, w)$-module $P$, there exists a subset of (distinct) vertices $W$ of $E^0$ and positive integers $\{n_w\mid w\in W\}$ for which $P\cong \bigoplus_{w\in W}(L_{\K}(E, w)w)^{n_w}$. The next proposition guarantees that there is an positive integer $n$ such that all the vertices of $\overline{W}$ appear in the presentation of $P^n$. Furthermore, if $P$ is  a generator, then $\overline{W}$ is the same to $E^0$.

\begin{prop} \label{progeneratorprop}
Let $(E, w)$ be a finite vertex weighted graph, $\K$ a field, and   $P$ a nonzero finitely generated projective left $L_{\K}(E, w)$-module. Then there exist a saturated hereditary subset $H$ of $E^0$ and positive integers $n$, $n_v $ $(v\in H)$ such that 
$$P^n\ \cong\ \bigoplus_{v\in H} \  (L_{\K}(E, w)v)^{n_v}.$$
Moreover, if $P$ is a generator for the category $L_{\K}(E, w) \Mod$ of left $L_{\K}(E, w)$-modules, then $H = E^0$.
\end{prop}
\begin{proof}
Since $P$ is a finitely generated projective left $L_{\K}(E, w)$-modules and by Theorem~\ref{wgraphmonoid}(2), there exist a nonempty subset $X\subseteq E^0$ and positive integers $m_v $ $(v\in X)$ such that   

\begin{equation*}P \cong \bigoplus_{v \in X}(L_{\K}(E, w)v)^{m_v}.  \tag{\mbox{$\ast$}}\end{equation*}
	
Let $\overline{X}$ be the smallest hereditary and saturated subset of $E^0$ containing $X$.    We claim that 
$$P^n \cong \bigoplus_{v\in \overline{X}} (L_{\K}(E, w)v)^{n_v},$$ 
where $n\ge 1$ and each $n_v \geq 1$.  Since $E^0$ is finite,  by Lemma \ref{h-sclosure}, it follows that $\overline{X}= G_{n_0}$ for some $n_0\ge 1$, where $G_0 = X$ and $G_{n} = H(G_{n-1}) \cup S(G_{n-1}) \cup G_{n-1}$ for all $n\ge 1$. For $z\in H(X)\setminus X$, there exist a vertex $u\in X$ and an edge $f\in s^{-1}(u)$ such that $z = r(f)$. Let $k_v$ be the smallest positive integer such that $k_um_u \ge w(u) +1$. We then have
\begin{equation*}P^{k_u} \cong \bigoplus_{v \in X}(L_{\K}(E, w)v)^{k_um_v}\cong (L_{\K}(E, w)u)^{k_um_u} \oplus\bigoplus_{v\in X \setminus \{u\}}(L_{\K}(E, w)v)^{k_um_v}.  \tag{\mbox{$\ast\ast$}}\end{equation*}
By Theorem \ref{wgraphmonoid}(1), it follows that
\begin{equation*}(L_{\K}(E, w)u)^{w(u)}\cong \bigoplus_{e\in s^{-1}(u)}L_{\K}(E, w)r(e)\cong L_{\K}(E, w)z \oplus \bigoplus_{e\in s^{-1}(u) \setminus \{f\}}L_{\K}(E, w)r(e),  \tag{\mbox{$\dagger$}}\end{equation*}
where $r(e)\in H(X)$ for all $e\in s^{-1}(u) \setminus \{f\}$.
\noindent
Now replace any one of the summands isomorphic to $(L_{\K}(E, w))^{w(u)}$ which appears in the  decomposition ($\ast$$\ast$)  of $P^{k_u}$  by the isomorphic version of $(L_{\K}(E, w))^{w(u)}$ given in $(\dagger)$; we then arrive at  a direct sum decomposition of $P^{k_u}$ which includes a summand isomorphic to $L_{\K}(E, w)z$, and which has not disappeared any given $L_{\K}(E, w)v$ which appeared in  decomposition ($\ast$$\ast$) of $P^{k_u}$. Continuing this same process now on the summand $L_K(E)x$ with $x\in H(X)\setminus (X \cup \{z\})$, we see that after at most $|H(X)\setminus X|$ steps we obtain that
\begin{equation*}P^{n'} \cong \bigoplus_{v \in H(X)\cup X}(L_{\K}(E, w)v)^{m'_v} \tag{\mbox{$\ast\ast\ast$}}\end{equation*}
where $n'\ge 1$ and each $m'_v\ge 1$. 

For $z\in S(X)$, we have $r(e)\in X$ for all $e\in s^{-1}(z)$, and 

\begin{equation*}
(L_{\K}(E, w)z)^{w(z)}\cong \bigoplus_{e\in s^{-1}(z)}L_{\K}(E, w)r(e)\quad\quad (\text{by Theorem } \ref{wgraphmonoid}(1)). 
\end{equation*}	
If $m'_{r(e)} =1$ for some $e\in s^{-1}(z)$, then we replace $P^{n'}$ by $P^{2n'}\cong \bigoplus_{v \in H(X)\cup X}(L_{\K}(E, w)v)^{2m'_v}$. Therefore, we may assume that $m'_{r(e)}\ge 2$ for all $e\in s^{-1}(z)$. Now replacing any one of the summands isomorphic to $\bigoplus_{e\in s^{-1}(z)}L_{\K}(E, w)r(e)$ which appears in the  decomposition ($\ast$$\ast$$\ast$)  of $P^{n'}$  by  $(L_{\K}(E, w)z)^{w(z)}$, we arrive at a direct sum decomposition of $P^{n'}$ which includes a summand isomorphic to $(L_{\K}(E, w)z)^{w(z)}$, and which has not disappeared any given $L_{\K}(E, w)v$ which appeared in  decomposition ($\ast$$\ast$$\ast$) of $P^{n'}$. Continuing this same process now on the summand $L_K(E)u$ with $u\in S(X)\setminus \{z\}$, we see that after at most $|S(X)|$ steps we obtain that
\begin{equation*}P^{n''} \cong \bigoplus_{v \in H(X)\cup S(X)\cup X}(L_{\K}(E, w)v)^{m''_v}= \bigoplus_{v\in G_1}(L_{\K}(E, w)v)^{m''_v} \end{equation*}
where $n''\ge 1$ and each $m''_v\ge 1$. 

Continuing this same process now on the set $G_1$, we see that after $n_0$ steps we obtain that $$P^n \cong \bigoplus_{v\in G_{n_0}} (L_{\K}(E, w)v)^{n_v}= \bigoplus_{v\in \overline{X}} (L_{\K}(E, w)v)^{n_v},$$ 
where $n\ge 1$ and each $n_v \geq 1$, showing the claim.  
  
Let $u$ be an arbitrary vertex of $E^0$. Since $P$ is a generator for $L_{\K}(E, w) \Mod$, $P^n$ is also a generator for $L_{\K}(E, w) \Mod$. Then, there exist a positive integer $s$  a split epimorphism $P^{sn} \to L_{\K}(E, w)u \to 0$; so there are maps $\varphi \in {\rm Hom}_{L_{\K}(E, w)}(P^{sn}, L_{\K}(E, w)u)$ and $\psi \in {\rm Hom}_{L_{\K}(E, w)}(L_{\K}(E, w)u, P^{sn})$ for which $\varphi \psi $ is the identity map on $L_{\K}(E, w)u$. In particular, we have  $u = \varphi\psi(u)$. 

We should note that $\Hom_{L_{\K}(E, w)}(L_{\K}(E, w)x, L_{\K}(E, w)y)\cong xL_{\K}(E, w)y$ as $\K$-vector spaces for all $x, y\in E^0$ (by \cite[Proposition 1.5(1)]{an:colpaofgalpa}). Using this, 
and the standard decomposition of maps to and from finite direct sums, the equation $u = \varphi\psi(u)$ yields elements $r_{v,j}$ and $r'_{v,j}$ in $L_{\K}(E, w)$, with $v\in \overline{X}$ and $1 \leq j \leq s\cdot n_v$, for which 
$$u =  \sum_{v\in \overline{X}} \sum_{j=1}^{s\cdot n_v} u r_{v,j} v r'_{v,j} u .$$

Since each $v$ is in $\overline{X}$, each term $u r_{v,j} v r'_{v,j} u$ is in the ideal $I(\overline{X})$ of $L_{\K}(E, w)$ generated by $\overline{X}$, and so $u \in I(\overline{X})$. This implies that $E^0\subseteq I(\overline{X})$. Then, by Theorem \ref{quotientwLPa}(2), it follows that $E^0 \subseteq I(\overline{X}) \cap E^0 = \overline{X}$, and therefore $\overline{X} = E^0$, thus finishing the proof.
\end{proof} 

The following example will help illuminate the ideas of Proposition \ref{progeneratorprop}.

\begin{example}\label{progeneratorexam} Let $\K$ be a field and $(E, w)$ the following weighted graph
$$E \ \ \ = \ \ \  
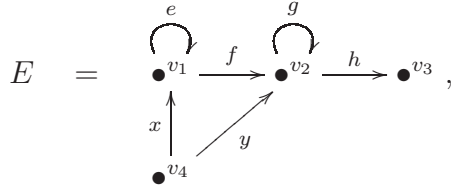
\xymatrix{  \bullet^{v_1}\ar@(ul,ur)^{e}\ar[r]^{f}&  \bullet^{v_2}\ar@(ul,ur)^{g}\ar[r]^{h}&\bullet^{v_3}\\ \bullet^{v_4}\ar[u]^{x}\ar[ur]_{y}},$$ where all the edges have weight $2$. We then have that $(E, w)$ is a vertex weighted graph. Consider the finitely generated projective left $L_{\K}(E, w)$-module $P = L_{\K}(E, w)v_1$ and $I := L_{\K}(E, w)v_1L_{\K}(E, w)$. Since $v_1\in I$, it follows that $f_1 = v_1f_1\in I$ and $f_2 =v_1f_2\in I$, and so $v_2 = f^*_1f_1 + f^*_2f_2 \in I$, $v_3 = h^*_1h_1 + h^*_2h_2= h^*_1v_2h_1 + h^*_2v_2h_2\in I$, and $v_4 = x_1x^*_1 + y_1y^*_1 = x_1v_1x^*_1 + y_1v_2y^*_1\in I$. Hence we have $I = L_{\K}(E, w)$, and therefore $v_1$ is a full idempotent in $L_{\K}(E, w)$, that means, $P$ is a generator for  $L_{\K}(E, w) \Mod$.

We have $P^2 \cong L_{\K}(E, w)v_1\oplus L_{\K}(E, w)v_2$ and $(L_{\K}(E, w)v_2)^2\cong L_{\K}(E, w)v_2 \oplus L_{\K}(E, w)v_3$, and so \[P^4\cong (L_{\K}(E, w)v_1)^2\oplus (L_{\K}(E, w)v_2)^2\cong (L_{\K}(E, w)v_1)^2\oplus L_{\K}(E, w)v_2\oplus L_{\K}(E, w)v_3.\] Since $(L_{\K}(E, w)v_4)^2\cong L_{\K}(E, w)v_1\oplus L_{\K}(E, w)v_2$, we have 
\begin{align*}
P^8&\cong  ((L_{\K}(E, w)v_1)^4\oplus (L_{\K}(E, w)v_2)^2\oplus (L_{\K}(E, w)v_3)^2\\
& \cong (L_{\K}(E, w)v_1)^3\oplus L_{\K}(E, w)v_2\oplus (L_{\K}(E, w)v_3)^2\oplus (L_{\K}(E, w)v_4)^2.
\end{align*}  
\end{example}

In order to describe corners of weighted Leavitt path algebras we need the following very useful notion introduced in \cite[Definition 3.3]{an:colpaofgalpa} for the case of finite graphs.

\begin{deff}(Hair extension of a weighted graph) \label{hairdef}    Let $(E, w)$ be a finite weighted graph with $E^0 = \{v_1, v_2, \dots, v_t\}$.   Let $n_1, n_2, \dots, n_t$ be a sequence of positive integers.  The {\it hair extension} $(E^+(n_1, n_2, \dots, n_t), w^+)$ of $(E, w)$ is the weighted graph formed form $(E, w)$ by adding these vertices and edges  to $E$:  
$$\xymatrix{  \bullet^{v_i^{n_i-1}} \ar[r]^{e_i^{n_i-1}} & \cdots \bullet^{v_i^2} \ar[r]^{e_i^2} & \bullet^{v_i^1} \ar[r]^{e_i^1} & }  \ $$ where  $r(e_i^1)=v_i$, and $w^+$ is an extension of $w$ such that $\bar{w}(e^j_i) =1$ for all $1\le i\le t$ and $1\le j\le n_i-1$. 
\end{deff}


For clarification, we consider the following example.
\begin{example}\label{hairex}
If $E$ is the graph
	
$$
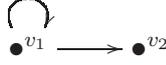
\xymatrix{\bullet^{v_1} \ar@(ul,ur)\ar[r]& \bullet^{v_2}}$$ then
$E^+(3,2)$ is the graph
$$
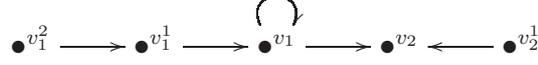
\xymatrix{\bullet^{v_1^2}\ar[r]& \bullet^{v_1^1}\ar[r]&
		\bullet^{v_1} \ar@(ul,ur)\ar[r]& \bullet^{v_2}&\bullet^{v_2^1}\ar[l]}$$ 
\end{example}

It is worth mentioning the following note.

\begin{remark}\label{hair-ext-subal}
Let $(E, w)$ be a weighted graph, $(E^+, w^+)$ a hair extension of $(E, w)$ and $\K$ a field. Then the map $L_{\K}(E, w)\longrightarrow L_{\K}(E^+, w^+)$, defined by $v\longmapsto v$ for all $v\in E^0$,  $e_i\longmapsto e_i$ and $e^*_i\longmapsto e^*_i$ for all $e\in E^1$ and $1\le i\le w(e)$, is an injective homomorphism of $\K$-algebras.
\end{remark}
\begin{proof}
The statement immediately follows from Corollary \ref{completewsubgraph}, as $(E, w)$ is a complete weighted subgraph of $(E^+, w^+)$. 	
\end{proof}	

We are now in a position to give the first main result of this article, describing the endomorphism ring of projective and finitely generated generators for the category of left $L_{\K}(E, w)$-modules.

\begin{theorem} \label{maintheo1}
Let $(E, w)$ be a finite vertex weighted graph, $\K$ a field, and   $P$ a projective and finitely generated generator for the category $L_{\K}(E, w) \Mod$ of left $L_{\K}(E, w)$-modules. Then there exist a positive integer $n$ and a hair extension $(E^+, w^+)$ of $(E, w)$ such that 
	\[\M_n(\textnormal{End}_{L_{\K}(E, w)}(P))\ \cong\ L_{\K}(E^+, w^+).\]	
\end{theorem} 
\begin{proof}
By Proposition \ref{progeneratorprop},  we have $$P^n \cong \bigoplus_{v\in E^0} (L_{\K}(E, w)v)^{n_v},$$
where $n\ge 1$ and each $n_v \geq 1$.   Write $E^0 = \{v_1, v_2, \dots , v_t\}$.  Note that there are $\sigma = \sum_{v\in E^0} n_v$ direct summands in the decomposition. By \cite[Proposition 1.5]{an:colpaofgalpa}, $\M_n(\textnormal{End}_{L_{\K}(E, w)}(P))\ \cong \text{End}_{L_{\K}(E, w)}(P^n)$ is isomorphic to a $\sigma \times \sigma$ matrix ring, with entries described as follows.  The indicated matrices   may be viewed as consisting of rectangular blocks of size $n_{v_i} \times n_{v_j}$, where, for $1 \leq i \leq t$, $1\leq j \leq t$, the entries of the $(i,j)$ block are elements of the $K$-vector space   $v_i L_{\K}(E, w) v_j$.    

On the other hand, because $n_{v_i} \geq 1$ for all $1\leq i \leq t$, we may construct  the hair extension $E^+ = E^+(n_{v_1}, n_{v_2}, \dots , n_{v_u})$ of $(E, w)$.    For each $1 \leq i \leq t$, and each $1\leq y \leq n_{v_i}$,  let $p_i^y:= e_i^y \cdots e_i^1$ denote the (unique) path in $E^+$ having  $s(p_i^y) = v_i^y$, and $r(p_i^y) = v_i$.    Note that $p_i^y (p_i^y)^* = v_i^y$ in $L_{\K}(E^+, w^+)$ and $|(E^+)^0| =   \sum_{1\leq i \leq u} n_{v_i}=\sigma$. Since $1_{L_{\K}(E^+, w^+)} = \sum_{v\in (E^+)^0}v$
and $vu= \delta_{vu}v$ in $L_{\K}(E^+, w^+)$ for all $v, u\in (E^+)^0$, it follows that $L_{\K}(E^+, w^+) = \oplus_{v\in (E^+)^0} L_{\K}(E^+, w^+)v$. By \cite[Proposition 1.5]{an:colpaofgalpa}, we obtain that $L_{\K}(E^+, w^+)$ is isomorphic to the $\sigma \times \sigma$ matrix ring with entries described as follows.   For $1 \leq i, j \leq t$, and $0 \leq y \leq n_{v_i}-1$, $0 \leq z \leq n_{v_j}-1,$ the entries in the row indexed by $(n_{v_i}, y)$ and column indexed by $(n_{v_j}, z)$ are elements of $v_i^y L_{\K}(E^+, w^+) v_j^z$, where $v^0_i := v_i$.   

We now show that these two $\sigma \times \sigma$ matrix rings are isomorphic as $\K$-algebras.  To do so, we show first that for each pair $(n_{v_i}, y)$,  $(n_{v_j}, z)$  with $1 \leq i, j \leq t$, and $1 \leq y \leq n_{v_i}-1$, $1 \leq z \leq n_{v_j}-1$, there is a $\K$-vector space isomorphism
$$\varphi = \varphi_{(n_{v_i}, y), (n_{v_j}, z)} :   v_i L_{\K}(E, w) v_j \rightarrow v_i^y L_{\K}(E^+, w^+) v_j^z.$$ 
For $r
\in L_{\K}(E, w)$ we define
$$ \varphi_{(n_{v_i}, y), (n_{v_j}, z)}(v_i r v_j) = p_i^yv_i r v_j (p_j^z)^*.$$
By Remark \ref{hair-ext-subal}, $L_{\K}(E, w)$ is a $\K$-subalgebra of $L_{\K}(E^+, w^+)$, and so $\varphi$ is a $\K$-vector space homomorphism. If $p_i^yv_i r v_j (p_j^z)^* = 0$, then multiplying on the left by $(p_i^y)^*$ and on the right by $p_j^z$ yields $v_i r v_j = 0$, and therefore $\varphi$ is injective. To show  $\varphi$ is surjective:  for $v_i^y s v_j^z \in v_i^y L_{\K}(E, w) v_j^z$ with $s\in L_{\K}(E^+, w^+)$, define $s' = (p_i^y)^*v_i^y  s v_j^z p_j^z \in v_i L_{\K}(E^+, w^+) v_j$.  But  using the fact that there are no paths from elements of $E^0$ to any of the newly added vertices which yield $E^+$, we have as above that $s'$ may be viewed as an element of $L_{\K}(E, w)$.    Then, using the previous observation that $p_i^y (p_i^y)^* = v_i^y$ in $L_{\K}(E^+, w^+)$, we conclude that $\varphi(s') = p_i^y (p_i^y)^* v_i^y s v_j^z p_j^z (p_j^z)^*  = v_i ^y s v_j^z$, and thus $\varphi$ is surjective.

We now define $\Phi$ to be the $\K$-space isomorphism between the two matrix rings induced by applying each of the $\varphi_{(n_{v_i}, y), (n_{v_j}, z)}$ componentwise.   We need only show that these componentwise isomorphisms respect the corresponding matrix multiplications.    But to do so, it suffices to show that the maps behave correctly in each component.    That is, we need only show, for each $m_\ell$ ($1\leq \ell \leq t$)   and each $x$ ($1 \leq x \leq m_{v_{\ell}}$), that 
$$ \varphi_{(n_{v_i}, y), (n_{v_j}, z)}(v_i r v_\ell) \cdot  \varphi_{(m_{v_{\ell}}, x), (m_j,z)}(v_\ell  r' v_j) =  \varphi_{(n_{v_i}, y), (n_{v_j}, z)}(v_i r v_\ell r'  v_j).$$ 
But this is immediate, as $(p_\ell^x)^* p_\ell^x = v_\ell$ for each $v_\ell \in E^0$ and $1 \leq x \leq m_{v_{\ell}}$, this finishing the proof.  	
\end{proof}	

For clarification, we illustrate Theorem \ref{maintheo1} by presenting the following example.

\begin{example}\label{progeneratorexam} Let $\K$ be a field and $(E, w)$ the following weighted graph
$$E \ \ \ = \ \ \  
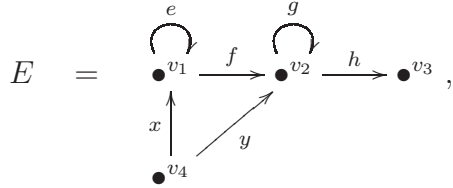
\xymatrix{  \bullet^{v_1}\ar@(ul,ur)^{e}\ar[r]^{f}&  \bullet^{v_2}\ar@(ul,ur)^{g}\ar[r]^{h}&\bullet^{v_3}\\ \bullet^{v_4}\ar[u]^{x}\ar[ur]_{y}},$$ where all the edges have weight $2$. Let $P = L_{\K}(E, w)v_1$. By Example \ref{progeneratorexam}, we have that $P$ is a  projective and finitely generated generator for  $L_{\K}(E, w) \Mod$, and $$P^8\cong (L_{\K}(E, w)v_1)^3\oplus L_{\K}(E, w)v_2\oplus (L_{\K}(E, w)v_3)^2\oplus (L_{\K}(E, w)v_4)^2.$$ By Theorem \ref{maintheo1}, we obtain that $$\M_8(\textnormal{End}_{L_{\K}(E, w)}(P))\ \cong\ L_{\K}(E^+(3, 1,2, 2), w^+).$$

\end{example}

As a consequence of Theorem \ref{maintheo1}, we obtain the following  interesting result.

\begin{corollary}\label{CornerofwLPAs}
Let $(E, w)$ be a finite vertex weighted graph, $\K$ a field, and   $\epsilon$ a full idempotent in  $L_{\K}(E, w)$. Then there exist a positive integer $n$ and a hair extension $(E^+, w^+)$ of $(E, w)$ such that 
\[\M_n(\epsilon L_{\K}(E, w)\epsilon)\ \cong\ L_{\K}(E^+, w^+).\]	
\end{corollary}
\begin{proof}
The statement follows from Theorem \ref{maintheo1}, as $L_{\K}(E, w)\epsilon$ ($\epsilon$ is a full idempotent in  $L_{\K}(E, w)$) is a projective and finitely generated generator for the category of left $L_{\K}(E, w)$-modules, and $\textnormal{End}_{L_{\K}(E, w)}(L_{\K}(E, w)\epsilon)\ \cong\ \epsilon L_{\K}(E, w)\epsilon$.
\end{proof}	

We note that in \cite[Theorem 3.15]{an:colpaofgalpa} Abrams and the second author proved that every nonzero corner of the Leavitt path algebra of a finite graph is isomorphic to a Leavitt path algebra. However, this is, in general, not true for the case of weighted Leavitt path algebras.

\begin{example}\label{CornerofwLPAsexample}
Let $\K$ be a field and $(E, w)$ the following weighted graph 
\begin{equation*}
\xymatrix{
	v_1  \ar@(lu,ld)_e\ar[r]^f & v_2, 
}
\end{equation*}	
where $w(e) = w(f) =2$. We then have that $v_1$ is a full idempotent in  $L_{\K}(E, w)$ (since $v_2 = f^*_1f_1 + f^*_2f_2$) and $(L_{\K}(E, w)v_1)^2 \cong L_{\K}(E, w)v_1 \oplus L_{\K}(E, w)v_2$. By Theorem \ref{maintheo1}, it follows that $\M_2(v_1 L_{\K}(E, w)v_1)\ \cong\ L_{\K}(E^+(1, 1), w^+) = L_{\K}(E, w)$. 

Assume that $v_1 L_{\K}(E, w)v_1$ is a weighted Leavitt path algebra. Since $v_1 L_{\K}(E, w)v_1$ is Morita equivalent to $L_{\K}(E, w)$, it follows that $\mathcal{V}(v_1 L_{\K}(E, w)v_1) \cong \mathcal{V}(L_{\K}(E, w))$. By Theorem~\ref{wgraphmonoid}, $ \mathcal{V}(L_{\K}(E, w)) \cong M_{(E,w)}$, and so
\[\mathcal{V}(v_1 L_{\K}(E, w)v_1)\cong M_{(E,w)} = \langle u, v\rangle/ \langle  2v=v+u\rangle.\]
It is easy to see that $M_{(E,w)}$ is an infinite commutative monoid.  On the other hand, by \cite[Proposition 40]{HR}, it follows that $L_{\K}(E, w)$ has a local valuation $\upsilon: L_{\K}(E, w)\longrightarrow \mathbb{N} \cup \{-\infty\}$ such that $0$ is the only element satisfying $\upsilon(0) = -\infty$. Let $x$ and $y$ be two elements in $v_1 L_{\K}(E, w)v_1$. We then have $\upsilon(xy)= \upsilon(x) + \upsilon(y)$ (see \cite[Definition 3.7]{HR}). Now if $xy$ is zero, and $x$ or $y$ is not zero, then the equality  $\upsilon(xy)= \upsilon(x) + \upsilon(y)$ would not happen. This shows that $v_1 L_{\K}(E, w)v_1$ is a domain. It is obvious that $v_1 L_{\K}(E, w)v_1$ is not a commutative ring. Then, by \cite[Theorem 41]{HR}, we have that $\mathcal{V}(v_1 L_{\K}(E, w)v_1) \cong  \langle v\rangle/ \langle  nv=mv\rangle$ which is either a finite cyclic monoid or a free cyclic monoid, a contradiction. Thus $v_1 L_{\K}(E, w)v_1$ is not a weighted Leavitt path algebra. 
\end{example}

\begin{deff}
Let $(E, w)$ be a row-finite weighted graph and $v\in E^0$ a source. We form the \emph{source elimination} weighted graph $(E_{\setminus v}, w_{r})$ of $(E, w)$ as follows: $E_{\setminus v}$ denotes the graph obtained from $E$ by deleting $v$ and all the edges in $E$ emitting from $v$, and $w_r$ is the restriction of the weight function $w$ to $E_{\setminus v}$.
\end{deff}

The following lemma is to extend \cite[Proposition 1.4]{alps:fiitcofpa} and \cite[Lemma 4.3]{ar:fpsmolpa} to the case of weighted Leavitt path algebras.

\begin{lemma}\label{sourceelim}
Let $\K$ be a field, $(E, w)$ a finite weighted graph and $v\in E^0$ a source and not a sink with $w(v) =1$. Then $L_{\K}(E, w)$ is Morita equivalent to 	$L_{\K}(E_{\setminus v}, w_r)$.
\end{lemma}
\begin{proof}
Since $(E_{\setminus v}, w_r)$ is a complete weighted subgraph of $(E, w)$ and by Corollary \ref{completewsubgraph}, 	the map $\phi: L_{\K}(E_{\setminus v}, w_r)\longrightarrow L_{\K}(E, w)$, defined by $u\longmapsto u$ for all $u\in E^0\setminus \{v\}$,  $e_i\longmapsto e_i$ and $e^*_i\longmapsto e^*_i$ for all $e\in E^1\setminus s^{-1}(v)$ and $1\le i\le w(e)$, is an injective homomorphism of $\K$-algebras.	

Let $\epsilon = \phi(1_{L_{\K}(E_{\setminus v}, w_r)}) = \sum_{u\in E^0, u\neq v}u$. We claim that $\phi(L_{\K}(E_{\setminus v}, w_r)) = \epsilon L_{\K}(E_{\setminus v}, w_r)\epsilon$ by following the proof of \cite[Lemma 4.3]{ar:fpsmolpa}. It is obvious that 
$\phi(L_{\K}(E_{\setminus v}, w_r)) \subseteq \epsilon L_{\K}(E, w)\epsilon$. 

Conversely, let $p = x_1\cdots x_n$ be a generalized path in $(E, w)$ such that $s(x_1)\neq v$ and $r(x_n)\neq v$, where $x_i\in \{e_i, e^*_i\mid e\in E^1, 1\le i\le w(e)\}$. Since $v$ is a source, $p$ cannot pass through $v$, and so $p$ is a generalized path in $E_{\setminus v}$. Consequently, $p = \phi(p)\in \phi(L_{\K}(E_{\setminus v}, w_r))$, showing the claim.

We next show that $L_{\K}(E, w) = L_{\K}(E, w)\epsilon L_{\K}(E, w)$. It suffices to prove that $v$ is in $L_{\K}(E, w)\epsilon L_{\K}(E, w)$. Since $w(v) =1$, it follows that $w(e) =1$ for all $e\in s^{-1}(v)$. Since $r(e)\in L_{\K}(E, w)\epsilon L_{\K}(E, w)$ for all $e\in s^{-1}(v)$, $e_1 = e_1r(e)\in L_{\K}(E, w)\epsilon L_{\K}(E, w)$ for all $e\in s^{-1}(v)$. We then have $$v = \sum_{e \in s^{-1}(v)}e_1e^*_1\in  L_{\K}(E, w)\epsilon L_{\K}(E, w),$$ which yields $L_{\K}(E, w) = L_{\K}(E, w)\epsilon L_{\K}(E, w)$. Therefore, $L_{\K}(E, w)$ is Morita equivalent to 	$L_{\K}(E_{\setminus v}, w_r)$, thus finishing the proof.
\end{proof}	

We are now in a position to give the second main result of this article, describing a unital $\K$-algbera being Morita equivalent to a weighted Leavitt path algebra.

\begin{theorem}\label{maintheo2}
Let $(E, w)$ be a finite vertex weighted graph, $\K$ a field, and   $A$ a unital $\K$-algebra. Then $A$ is Morita equivalent to $L_{\K}(E, w)$ if and only if there exist a positive integer $n$ and a hair extension $(E^+, w^+)$ of $(E, w)$ such that 
$\M_n(A)\ \cong\ L_{\K}(E^+, w^+).$		
\end{theorem}
\begin{proof}
Assume that $A$	is Morita equivalent to $L_{\K}(E, w)$. We then have $A \cong \epsilon \M_m(L_{\K}(E, w))\epsilon$ for some $m\ge 1$ and a full idempotent $\epsilon\in \M_m(L_{\K}(E, w))$. By Theorem \ref{matrixalgeoverwLPA}, $\M_m(L_{\K}(E, w))\cong L_{\K}(\M_mE, \bar{w})$ as $\K$-algebras. Since $(\M_mE, \bar{w})$ is a finite vertex weighted graph and by Corollary \ref{CornerofwLPAs}, there exist a positive integer $n$ and a hair extension $(E^+, w^+)$ of $(\M_mE, \bar{w})$ (consequently, a hair extension of $(E, w)$) such that  $\M_n(\epsilon L_{\K}(\M_mE, \bar{w})\epsilon)\ \cong\ L_{\K}(E^+, w^+).$ This implies that $\M_n(A)\ \cong\ L_{\K}(E^+, w^+).$ 

Conversely, suppose there exist a positive integer $n$ and a hair extension $(E^+, w^+)$ of $(E, w)$ such that 
$\M_n(A)\ \cong\ L_{\K}(E^+, w^+).$ Since $A$ is Morita equivalent to $\M_n(A)$, it follows that $A$ is Morita equivalent to $L_{\K}(E^+, w^+)$. We note that all newly added vertices which yield $E^+$ are sources and regular vertices of weight $1$. So we may apply the source elimination process by step-by-step at these vertices. At each step, the Morita equivalence of the corresponding weighted Leavitt path algebras is preserved by Lemma \ref{sourceelim}. Hence $L_{\K}(E^+, w^+)$ is Morita equivalent to $L_{\K}(E, w)$, as desired.
\end{proof}	

In \cite[Theorem 6.1]{GeneRooz} Abrams and the first author showed a nice connection between sandpile momoids and weighted Leavitt path algebras that any conical sandpile monoid $\text{SP}(E)$ of a directed sandpile graph $E$ can be realised as the $\mathcal{V}$-monoid of a weighted Leavitt path
algebra $L_{\K}(F, w)$, where $F$ is an explicitly constructed subgraph of $E$. Motivated by this result, they introduced sandpile algebras. As a consequence of Theorem \ref{maintheo2}, we next describe unital algebras being Morita equivalent to sandpile algebras. To do so, we need to recall some useful notions.

A finite graph $E$ is called a {\it sandpile graph} if $E$ has a unique sink (denote it by $s$), and for every $v\in E^0$ there exists a path $p$ with $s(p) = v$ and $r(p) = s$. Any sandpile graph $E$ may be given the structure of a balanced weighted graph $(E, w)$ by assigning $w(v) = |s^{-1}(v)|$ for all $v\in E_{\reg}^0$. A sandpile graph $E$ is called {\it conical} if any vertex not connected to a cycle has weight one. A unital algebra $A$ over a field $\K$ is called a {\it sandpile algebra} if there
exists a conical sandpile graph $E$ for which $$A \cong L_{\K}(E/S_E, w_r),$$ where $S_E$ is the set of all vertices in $E$ which do not connect to any cycle in $E$, and $w_r$ is the restriction of the weight function $w$ to $E/S_E$  (see \cite[Definition 6.11]{GeneRooz}).

\begin{cor}\label{sandpilealgebra}
Let $\K$ be a field and $A$ a unital $\K$-algebra. Then $A$ is Morita equivalent to a sandpile $\K$-algebra if and only if there exist a positive integer $n$, a conical sandpile graph $E$ and a hair extension $(G, w_G)$ of the vertex weighted graph $(E/S_E, w_r)$ such that $$\M_n(A)\cong L_{\K}(G, w_G).$$	
\end{cor}
\begin{proof}
The statement immediately follows from Theorem \ref{maintheo2}.
\end{proof}

\section{Acknowledgements}
The first author would like to wholeheartedly thank Vietnam Institute for Advanced Study Mathematics (VIASM) and his host T.G. Nam for the warm hospitality. He also acknowledges Australian Research Council grant DP230103184. The second author was supported by the Vietnam Academy of Science and Technology under grant
CTTH00.01/24-25.



\end{document}